\documentclass[12pt,a4paper]{amsart} 
\usepackage[utf8]{inputenc}
\usepackage[english]{babel}
\usepackage{amsmath}
\usepackage{amsfonts}
\usepackage{amssymb}
\usepackage{mathrsfs}
\usepackage{color}

\usepackage{amsthm}
\newtheorem{thm}{Theorem}%[section]
\newtheorem{cor}{Corollary}%[section]

\theoremstyle{definition}

\usepackage{graphics}
\usepackage{epstopdf}
\usepackage{epsfig}

\usepackage{url}
\usepackage{hyperref}

\usepackage[a4paper,top=2.5cm,bottom=2.8cm,
left=3.2cm,right=3.2cm]{geometry}

\newcounter{num} 
\newcommand{\Fg}[1][]{\thenum}

\newcommand\Al{\alpha}
\newcommand\B{\beta}
\newcommand\De{\delta}
\newcommand\G{\gamma}
\newcommand\Ep{\varepsilon}

\newcommand\CC{\mathcal{C}}

\markleft{\hfill \textsc{A generalization of Pappus chain theorem} \hfill}
\begin{document}
%% ===================
%Sangaku Journal of Mathematics (SJM) \copyright SJM \\
%ISSN 2534-9562 \\
%Volume 2 (2018) pp. 13-16 \\
%Received 22 May 2018. Published on-line 30 May 2018 \\ 
%web: \url{http://www.sangaku-journal.eu/} \\
%\copyright The Author(s) This article is published 
%with open access\footnote{This article is distributed under the terms of the Creative Commons Attribution License which permits any use, distribution, and reproduction in any medium, provided the original author(s) and the source are credited.}. \\
% ===========================   
\bigskip
%\bigskip

\begin{center}
{\Large \textbf{A generalization of Pappus chain theorem}} \\
\medskip
\bigskip
\textsc{Hiroshi Okumura} \\
Takahanadai Maebashi Gunma 371-0123, Japan\\
e-mail: \href{mailto:hokmr@yandex.com}{hokmr@yandex.com}\\

\end{center}
\bigskip

% ============================== 
\noindent
\textbf{Abstract.} We generalize Pappus chain theorem and give an analogue 
to this theorem. 



\medskip

\section{Introduction}

Let $\Al$, $\B$ and $\G$ be circles with diameters $BC$, $CA$ and $AB$, 
respectively for a point $C$ on the segment $AB$. % (see Figure \ref{fc}). 
Pappus chain theorem says: if $\{\Al=\De_0, \De_1, \De_2,\cdots \}$ is 
a chain of circles whose members touch $\B$ and $\G$, the 
distance between the center of the circle $\De_n$ and the line $AB$ equals 
$2nr_n$, where $r_n$ is the radius of $\De_n$ (see Figure \ref{far}). In this 
article we give a simple generalization of this theorem and show that if we 
consider a line passing through the centers of two circles in the chain 
instead of $AB$, a similar theorem still holds. 

\begin{figure}[h!]
	\centering
	\includegraphics[width=0.4\linewidth]{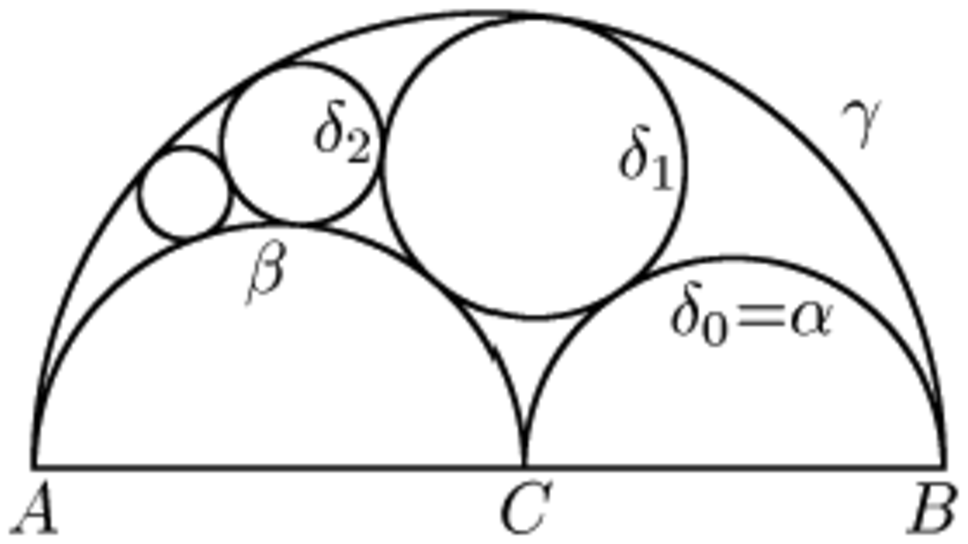}
        \refstepcounter{num}\label{far}\\
Figure \Fg .
\end{figure}
%\medskip

%%%%%%%%%%%%%%%%%%%%%%%%%%%%%%%%%%%%%%%%%%%%%%%%%%%%%%%%%%%%
\section{A generalization of Pappus chain theorem}
Let $\{\Ep_1, \Ep_2,\Ep_3\}=\{ \Al, \B,\G\}$ and $\{P_1,P_2,P_3\}=\{A,B,C\}$, 
where $P_3P_1$ and $P_1P_2$ are diameters of $\Ep_2$ and $\Ep_3$, respectively. 
We consider the chain of circles $\CC=\{\cdots, \De_{-2}, \De_{-1},\Ep_1=\De_0, 
\De_1, \De_2, \cdots \}$ whose members touch the circles $\Ep_2$ and $\Ep_3$. 
%where we assume that $\De_i$ lies on the region $y>0$ if $i>0$. 
Let $r_n$ be the radius of $\De_n$. Pappus chain theorem is obtained in the case 
$i=0$ in the following theorem (see Figure \ref{fcl4}).

\begin{thm} \label{t1} 
If $D_i$ is the center of the circle $\De_i\in\CC$ and $H_i(n)$ is the point of 
intersection of the line $P_1D_i$ and the perpendicular to $AB$ from $D_n$, 
the following relation holds.
\begin{equation}\label{eq1} 
|D_nH_i(n)|=2|n-i|r_n. 
\end{equation} 
\end{thm}

\begin{proof} 
We invert the figure in the circle with center $P_1$ orthogonal to $\De_n$. 
Then $\De_n$ and $P_1D_i$ are fixed and $\Ep_2$ and $\Ep_3$ are inverted to the 
tangents of $\De_n$ perpendicular to $AB$. Let $F$ be the foot of perpendicular 
from $D_n$ to $AB$. Since $H_i(n)$ is the center of the image of $\De_i$, we 
have $|H_i(n)F|=2ir_n$, while $|D_nF|=2nr_n$. Hence we get 
\eqref{eq1}. 
\end{proof}

\begin{figure}[h!]
	\centering
	\includegraphics[width=0.65\linewidth]{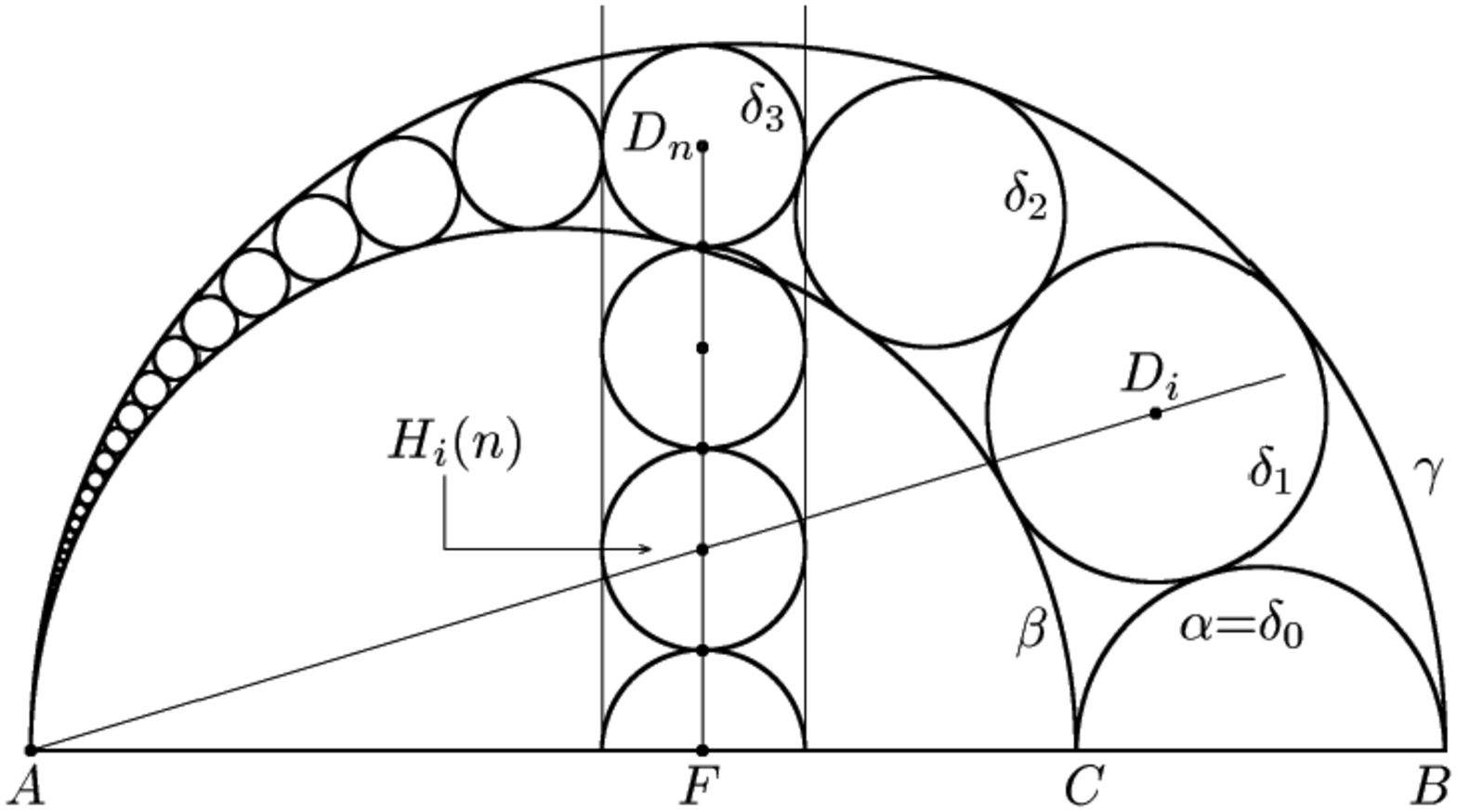}
        \refstepcounter{num}\label{fcl4}\\
Figure \Fg : $\CC=\CC_{\Al}$, $i=1$, $n=3$
\end{figure}

%%%%%%%%%%%%%%%%%%%%%%%%%%%%%%%%%%%%%%%%%%%%%%%%%%%%%%
\section{An analogue to Pappus chain theorem}

Let $a$ and $b$ be the radii of the circles $\Al$ and $\B$, respectively. 
We use a rectangular coordinate system with origin $C$ such that $A$ and $B$ 
have coordinates $(-2b,0)$ and $(2a,0)$, respectively. If $\Ep_1=\Al$, 
the chain is explicitly denoted by $\CC_{\Al}$. The chains $\CC_{\B}$ and 
$\CC_{\G}$ are defined similarly. Let $c=a+b$ and let 
$(x_n,y_n)$ be the center coordinate of the circle $\De_n\in\CC$. We have 
$y_n=2nr_n$ 
by Pappus chain theorem, and $x_n$ and $r_n$ are given in Table 1 \cite{LC1,LC2}.

%\medskip
%\begin{figure}[h!]
%	\centering
%	\includegraphics[width=0.65\linewidth]{c.eps}
%	\refstepcounter{num}\label{fc}\\
%	\caption{$\CC_{\G}$, $\Ep_1=\G$, $\{\Ep_2,\Ep_3\}=\{\Al,\B\}$}
%\end{figure}

%\medskip
\begin{center}
\begin{table}[htbp]
%\begin{center}
\begin{tabular}{|c|c|c|} \hline
Chain  & $x_n$  & $r_n$    \\ \hline\hline 
$\CC_{\Al}$&$-2b+\dfrac{bc(b+c)}{n^2a^2+bc}$&$\dfrac{abc}{n^2a^2+bc}$ \\ \hline
$\CC_{\B}$ &$2a-\dfrac{ca(c+a)}{n^2b^2+ca}$ &$\dfrac{abc}{n^2b^2+ca}$\\ \hline
$\CC_{\G}$ &$\dfrac{ab(b-a)}{n^2c^2-ab}$&$\dfrac{abc}{n^2c^2-ab}$\\ \hline
\end{tabular} \\
\vskip3mm
Table 1: $y_n=2nr_n$
%\end{center}
\end{table}
\end{center}
\vskip-5mm

Let $l_{i,j}$ $(i\not=j)$ be the line passing through the centers of the circles 
$\De_i$ and $\De_j$ for $\De_i,\De_j\in\CC$. It is expressed by the equations 
\begin{eqnarray}\label{eq3}
\left\{
\begin{array}{l}
2(bc-a^2ij)x+a(b+c)(i+j)y-2b(2a^2ij-c(b-c))=0,  \\
2(ca-b^2ij)x-b(c+a)(i+j)y+2a(2b^2ij+c(c-a))=0,  \\
2(ab+c^2ij)x+c(a-b)(i+j)y-2ab(a-b)=0 
\end{array}
\right.
\end{eqnarray}
in the cases $\CC=\CC_{\Al}$, $\CC=\CC_{\B}$, $\CC=\CC_{\G}$, respectively. 

Let $H_{i,j}(n)$ be the point of intersection of the lines $l_{i,j}$ and 
$x=x_n$ with $y$-coordinate $h_{i,j}(n)$. Let $d_{i,j}(n)=h_{i,j}(n)-y_n$, 
i.e., $d_{i,j}(n)$ is the signed distance between the center of $\De_n$ and 
$H_{i,j}(n)$. 
The following theorem is an analogue to Pappus chain theorem (see Figure \ref{fcl3}). 
It is also a generalization of \cite{ALTOK}.

\begin{thm} \label{t1} 
If $i+j\not=0$, then $d_{i,j}(n)=f_{i,j}(n)r_n$ holds, where 
\begin{equation}\label{eqf}
f_{i,j}(n)=\frac{2(n-i)(n-j)}{i+j}. 
\end{equation}
\end{thm}

\begin{proof}
We consider the chain $\CC_{\Al}$. By Table 1 and \eqref{eq3}, we get 
$$
h_{i,j}(n)=\frac{2(n^2+ij)abc}{(i+j)(n^2a^2+bc)}=2\frac{(n^2+ij)}{(i+j)}r_n. 
$$
Therefore 
$$
d_{i,j}(n)=h_{i,j}(n)-y_n=2\frac{(n^2+ij)}{(i+j)}r_n-2nr_n
%=\left|\frac{2(n-i)(n-j)abc}{(i+j)(a^2n^2+bc)}\right|
=\frac{2(n-i)(n-j)}{i+j}r_n. 
$$
The rest of the theorem can be proved in a similar way. 
\end{proof}

\begin{figure}[h!]
	\centering
	\includegraphics[width=.85\linewidth]{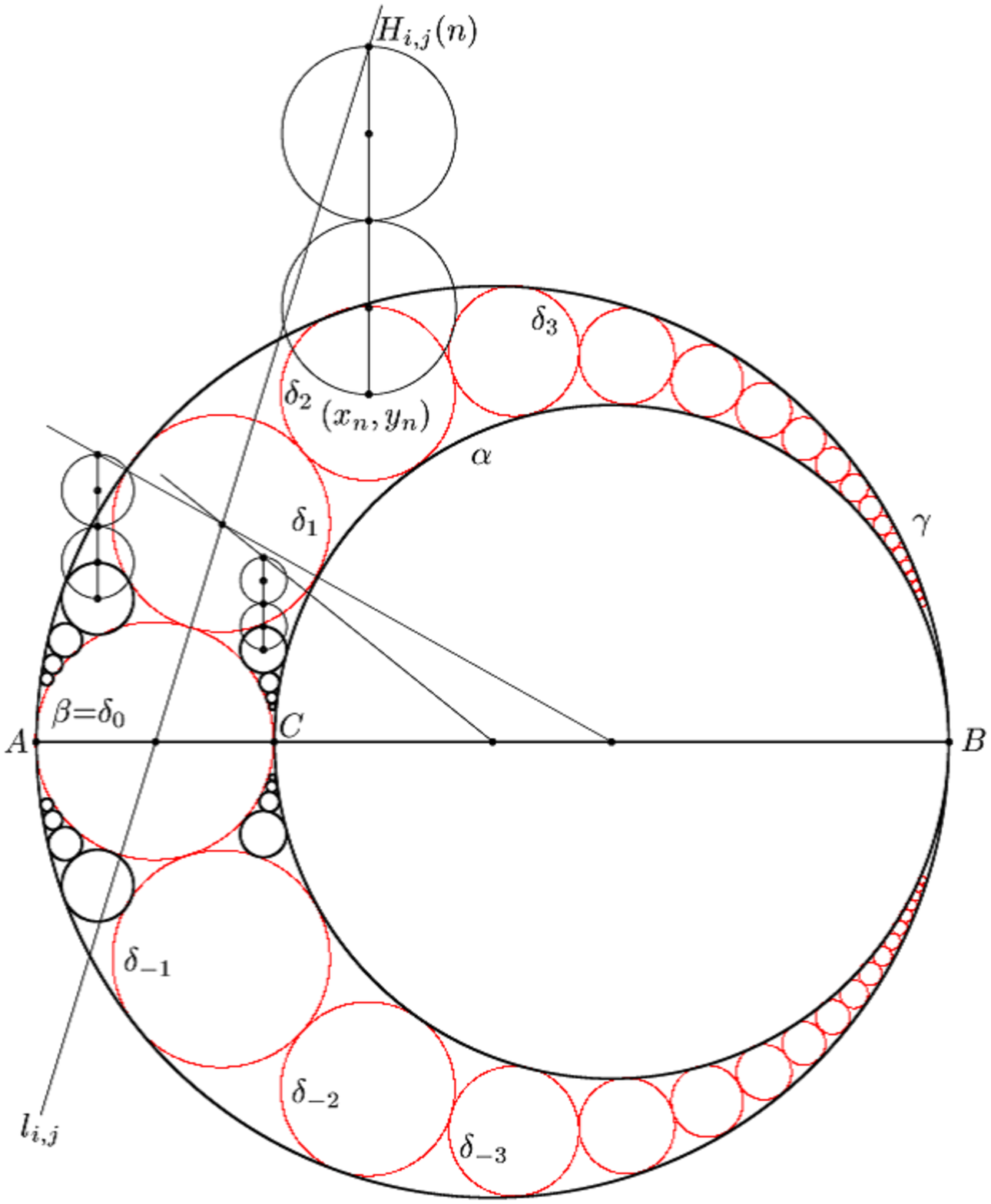}
	\refstepcounter{num}\label{fcl3}\\
Figure \Fg : $\CC=\CC_{\B}$, $\{i,j\}=\{0,1\}$, $n=2$
\end{figure}

\begin{cor}\label{c1}
If $i=0$ in Theorem \ref{t1}, the following statements hold. \\
\noindent{\rm (i)} If $j=\pm1$, $d_{i,j}(n)=\pm 2n(n\mp1)r_n$. \\
\noindent{\rm (ii)} If $j=\pm2$, $d_{i,j}(n)=\pm n(n\mp2)r_n$. \\
\end{cor}

\begin{cor}\label{c2}
$d_{i,j}(n)-d_{i,j}(-n)=-4nr_n$ for any integers $i$, $j$, $n$ with $i\not=\pm j$.
\end{cor}

\end{document}